\documentclass[a4paper, 11pt] {article}
\usepackage{amsfonts,amsmath}
\usepackage{amsthm}

\usepackage{graphicx}

\numberwithin{equation}{section}
\newtheorem{thm}{Theorem}[section]
\newtheorem{example}[thm]{Example}

\newcommand{\bsa}{\boldsymbol{a}}
\newcommand{\bsb}{\boldsymbol{b}}
\newcommand{\bsu}{\boldsymbol{u}}

\newcommand{\bsx}{\boldsymbol{x}}
\newcommand{\bsz}{\boldsymbol{z}}
\newcommand{\rd}{\mathrm{d}}

\title{On approximating the shape of  one dimensional functions}
\author{Chaitanya Joshi, Paul T. Brown and Stephen Joe}
\begin{document}
\maketitle

%
%
%

\begin{abstract}

Consider an $s$-dimensional function being evaluated at $n$ points of a low discrepancy sequence (LDS), where the objective is to approximate the one-dimensional functions that result from  integrating out $(s-1)$ variables. Here, the emphasis is on accurately approximating the shape of such \emph{one-dimensional} functions. Approximating this shape when the function is evaluated on a set of grid points instead is relatively straightforward. However, the number of grid points needed increases exponentially  with $s$.  LDS are known to be increasingly more efficient at integrating $s$-dimensional functions compared to grids, as $s$ increases. Yet, a method to approximate the shape of a one-dimensional  function when the function is evaluated using an $s$-dimensional LDS has not been proposed thus far. We propose an approximation method for this problem. This method is based on an $s$-dimensional integration rule together with fitting a polynomial smoothing function. We state and prove results showing conditions under which this polynomial smoothing function will converge to the true one-dimensional function. We also demonstrate the computational efficiency of the new approach compared to a grid based approach.
\end{abstract}



\section{Introduction}\label{intro}

While the focus of this paper is purely mathematical, we start by first outlining the motivation behind solving this particular problem and providing the context for the specific examples discussed.

\subsection{Motivation} \label{motivate}
This work is motivated by an application in Bayesian statistical inference where there is an interest in the one-dimensional posterior distributions. While, Monte Carlo based methods such as the Markov Chain Monte Carlo or the Approximate Bayesian Computation  are more widely used to approximate posterior distributions, these can be computationally expensive. Methods that instead explore the posterior distribution using a deterministic set of points  --- for example, using grid points \cite{rue09} and using central composite design (CCD) points  \cite{martin13} --- have been proposed as computationally efficient alternatives. However, since the number of grid points increases exponentially with $s,$ grid based methods can only be used when the (hyper) parameter space has very few dimensions  \cite{rue09}. Using CCD points is more efficient however, finding one dimensional distributions is then not straightforward. Existing numerical integration free methods can only approximate uni-modal distributions \cite{martin13}. Therefore, there is potential to explore the use of LDS to approximate the posterior distributions instead since such approximations could be more computationally efficient as well as accurate compared to those obtained using grid points or CCD points. However, as yet, there is no method to approximate one dimensional marginals using an LDS.\\

In this paper we develop a method to approximate the shape of the one-dimensional functions when an $s$-dimensional function is evaluated using $N$ LDS points. However, the focus of this paper is purely mathematical. It is not expected that the method developed here can be used to approximate Bayesian posterior distributions in its existing form. We expand more on this point in Section \ref{discussion}. In this paper we simply develop a method and prove the convergence theorems for the approximations.

\subsection{Integration Rules and Low Discrepancy Sequences}

Suppose we have an integrable function $g\,:\,[\bsa,\bsb] \rightarrow \mathbb{R}$, where $\bsa=(a_1,\ldots,a_s)$ and $\bsb=(b_1,\ldots,b_s)$ with $a_j< b_j$ for $1\le j\le s$. Without loss of generality, we may take the region of interest to be the unit hypercube $[0,1]^s$ since a linear transformation may be used to map a function~$g$ defined over $[\bsa,\bsb]$ to
a function $f$ defined over $[0,1]^s$. \\



\noindent Now consider the $s$-dimensional integral
\begin{equation*} \label{lds1}
I =\int_{[0,1]^{s}} f(\bsx) \,\rd\bsx.
\end{equation*}
The standard approach taken to find an approximation to $I$ is typically to make use of an integration rule. These integration rules are of the form 
\begin{equation}\label{intrule}
\hat{I}_{N} = \frac{1}{N}\sum_{i=1}^{N} f(\bsx_i),
\end{equation}
where the points $\bsx_{1},\ldots,\bsx_{N}$ are
sampled from the unit hypercube $[0,1]^{s}$. There are a number of choices for the integration rules. One can
use Monte Carlo (MC) rules in which the points are chosen randomly. However,
such points do suffer from large gaps and clusters and this can affect the
accuracy of the estimate for a given set of points \cite{lemieux09}. If the
point set was taken to be the regular $n$-point grid for which the
point set consists of the points $((i_1-1)/(n-1),(i_2-1)/(n-1),\ldots,
(i_s-1)/(n-1))$, where $1\le i_\ell \le n$ for $1\le \ell\le s$, then the
total number of points is $N=n^s$. If $s$ is large, then the number
of points increases rapidly as $n$ increases.

A large class of integration rules is the class of \emph{quasi-Monte Carlo}
(QMC) rules. These are equal weight integration rules of the form
(\ref{intrule}) that use deterministic point sets, specifically, the low discrepancy sequences (LDS).  These point sets have low \emph{discrepancy} with respect to the Lebesgue measure on a unit hypercube. One of the most commonly used discrepancy measure is called the star discrepancy. Let  $\mathcal{P}_{N}$ be an $N$ element point set in $[0,1]^{s}.$ For $\bsa  \in (0,1]^{s},$  the star discrepancy $D_{N}^{*}$ of this point set is defined as 
\begin{equation*}
D_{N}^{*} = \sup_{\bsa \in[0,1]^{s}} \left | \frac{\alpha([0, \bsa),\mathcal{P}_{N},N)}{N} - \prod_{j=1}^{s}a_{j} \right |,
\end{equation*}
where, $\alpha( [ 0, \bsa ) ,\mathcal{P}_{N} ,N) =$ \# $ \{n \in \mathbb{N}: 1 \leq n \leq N, \bsx_{n} \in [0, \bsa )\}.$ For an infinite sequence $\mathcal{P},$ the star discrepancy $D_{N}^{*}$ is the discrepancy of the first $N$ elements of $\mathcal{P}.$  A sequence of points is said to be \emph{low discrepancy} if $D_{N}^{*} \in O(N^{-1} \log(N)^{s}).$ The widely stated \emph{Koksma-Hlawka theorem} states that if the function $f$ has a variation $V(f)$ in the sense of \emph{Hardy and Krause} that is finite, then we have that $|I - \hat{I}_{N}| \leq V(f) D_{N}^{*}.$  For a general introduction to LDS, QMC rules and their applications, refer to \cite{leobacher14}, \cite{dick10} or \cite{lemieux09}. In this paper, the main QMC rules used in the numerical experiments are rank-$1$ lattice rules. These are rules in which 
\begin{equation}\label{gen} 
\bsx_i=\left\{\frac{i\bsz}{N}\right\},\quad 1\le i\le N.
\end{equation}
Here the $s$ components of $\bsz$ are integers in $\{1,2,\ldots,N-1\}$ and $\{x\}=x-\lfloor x\rfloor$ denotes the fractional part of $x\in\mathbb{R}$ which is applied component-wise for vectors. Although these are finite point sets and not sequences, the convergence rate of $O(N^{-1} \log(N)^{s})$ is still guaranteed (see, \cite{leobacher14}) . More information about lattice rules is also available in \cite{nied92} or \cite{sloan94}.\\


The three types of point sets that we discuss in this paper (grids, random points, LDS) can all be described using a common general description that we give below.\\

\noindent \textbf{ The point set $\mathcal{P}_{N}$:}\newline
\noindent In (\ref{intrule}), let the components of each $\bsx_i$ be denoted by $x_{i,j}$ for $1\le j \le s$. Let us now assume that for a fixed $j$ and $\forall i = 1,\ldots, N,$ there are $n$ distinct values of $x_{i,j}$ which we denote by $z_{k}$ for $1\le k\le n$. Here, for simplicity of notation, we have not included a $j$ subscript. Further, let us assume that there are exactly $m$ points that have the value $z_k$ for their $j^{th}$ subscript, for each $k, \, 1 \le k \le n.$
So  the total number of points $N$ satisfies $N=nm$.  Note that this description of point sets, which, from now on, we refer to as $\mathcal{P}_{N}$, in fact, covers a number of point sets including random points used for the MC integration rule. In particular, it includes an $n$-point grid and the rank-$1$ lattice rule  shown in Figure \ref{fig1b}. As seen in Figure \ref{fig1b} [a], in an $n-$ point regular grid, the points are aligned in rows and columns, each containing $n$ points. As a result, there are $n$ distinct $z_{k}$'s along each axis and $m=n.$ On the other hand as illustrated in   Figure \ref{fig1b} [b], in a rank-$1$ lattice, none of the points are aligned resulting in $n=N$ distinct $z_{k}$'s along each axis and $m=1$.\\

\begin{figure} [t]
\includegraphics[scale=0.4]{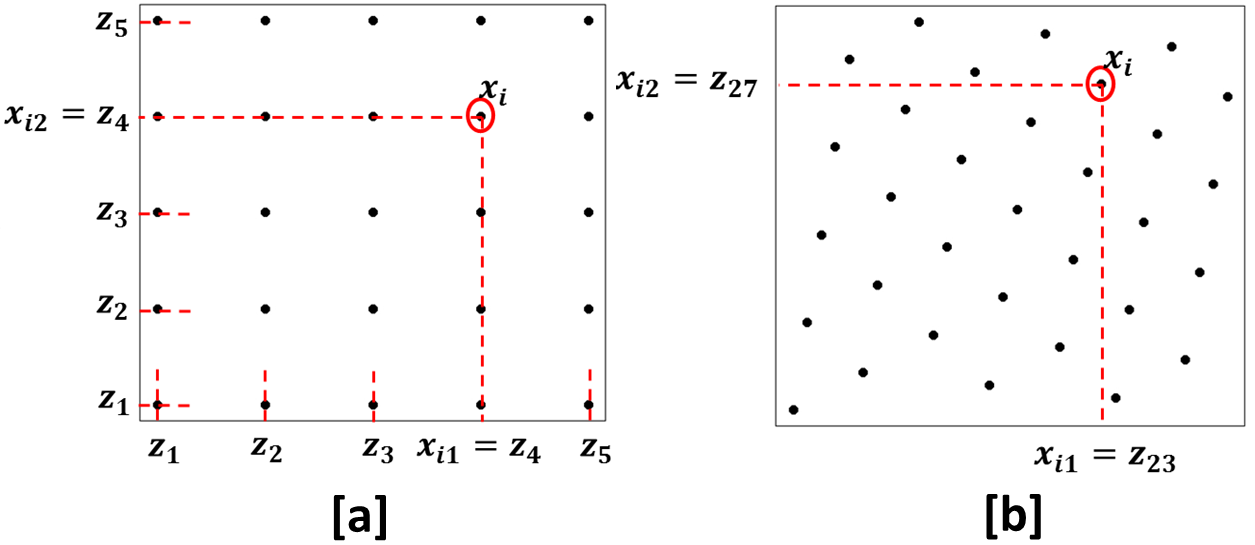}
\caption[]{[a] $5-$point grid ($m=5$) and [b] $32-$point rank-$1$ lattice ($m=1$). }
\label{fig1b}
\end{figure}

\subsection{Approximation to the one-dimensional functions using deterministic point sets} \label{1dapprox}

\begin{figure} [ht]
\includegraphics[scale=0.43]{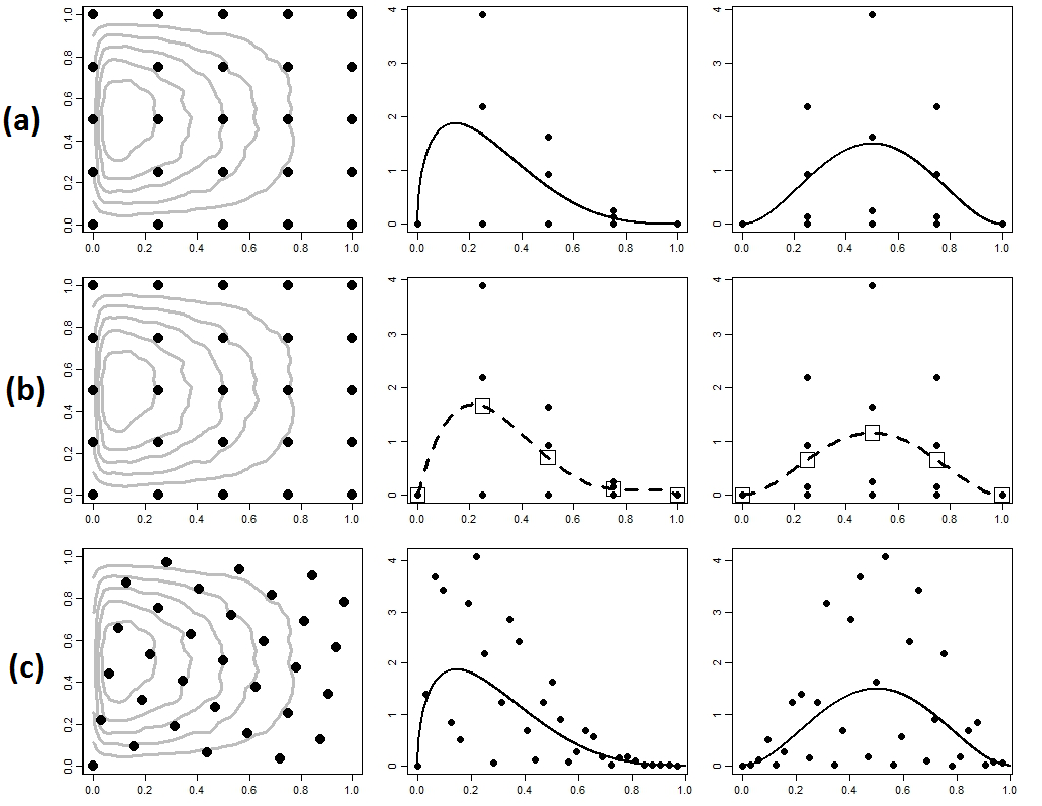}
\caption[]{\emph{First column:} Bi-variate Beta distribution contours along with the points used to approximate the one-dimensional functions:  (a),(b) $5-$point grid ($m=5$) and (c) $32-$point rank-$1$ lattice ($m=1$). \emph{Second and third columns}: the true-one dimensional functions along with the \emph{unique} orthogonal projections of the bi-variate Beta distribution for a $5-$point grid ($m=5$) in Row (a)  and a $32-$point rank-$1$ lattice ($m=1$) in Row  (c). Row (b) shows an interpolant fit through the point-wise means (squares) for the $5-$point grid. Because some of the function projections are the same, the number of points in the function projections are fewer than the total number of points shown in the first column on which the bi-variate function is evaluated. }
\label{fig1}
\end{figure}

Suppose that we are interested in approximating the functions
\begin{equation*}\label{fprob}
        f_j(x)=\int_{[0,1]^{s-1}}f(x_1,\ldots,x_{j-1},x,x_{j+1},\ldots,x_s)\,\rd\bsx_{-j},\quad x\in[0,1],
\end{equation*}
where, for a vector $\bsu=(u_1,\ldots,u_s),\, \bsu_{-j}$ denotes $(u_1,\ldots,u_{j-1},u_{j+1},\ldots,u_s),$ for $1\le j\le s$. So $f_j$ is the function obtained by integrating out all the variables of $f$ except the $j$-th one. The set of points $\{\bsx_i,\ 1\le i\le N\}$ could be obtained either by sampling randomly (MC approach) or using a $n$-point regular grid or using a QMC approach. An integration rule of the form (\ref{intrule}) can be used to approximate the one dimensiona functions. However, note that this approach does not approximate the \emph{shape} of the one dimensional function. By \emph{shape} we mean the graph of the one dimensional function (see Figure \ref{fig1}, columns 2,3, (a),(c)).\\

\begin{example}\label{grid}
As mentioned previously, the regular $n$-point grid consists of the points $((i_1-1)/(n-1),(i_2-1)/(n-1),\ldots,
(i_s-1)/(n-1))$, where $1\le i_\ell \le n$ for $1\le \ell\le s$. For the $j$-th coordinate of these $N=n^s$ points, we have
the $n$ distinct values $(i_j-1)/(n-1)$, $1\le i_j\le n$. As $N=n^s=nm$, it follows that $m=n^{s-1}$.
\end{example}

\begin{example}\label{maximal}
As mentioned previously, the points of an $\ell$-point rank-$1$ lattice rule are given by
$\{i\bsz/\ell\}$, where $\bsz\in\{1,2,\ldots \ell-1\}^s$. Now let $r$ be relatively prime with $\ell$. Then one can obtain the lattice rule with point set given by
\[
\left\{\frac{i\bsz}{\ell}+\frac{(k_1,k_2,\ldots,k_s)}{r}
\right\},\quad 1\le i\le \ell,\ 0\le k_1,k_2,\ldots,k_s\le r-1.
\]
Such a lattice rule has $N=\ell r^s$ points and is an example of a maximal rank lattice rule (for example, see \cite{sloan94}). Assuming
that all the components of $\bsz$ are relatively prime with $\ell$, then it may be shown that the $j$-th coordinate of these $N$ points consists
of the $n=\ell r$ distinct values $(i-1)/(\ell r)$ for $1\le i\le \ell r$ with each value repeated $m=r^{s-1}$ times. We note that in the $r=1$ case, the lattice rule is just a rank-$1$ lattice rule having a total of $\ell = N$ points. Moreover, the $j$-th coordinate of these points has the $N$ distinct values $z_k=(k-1)/N$ for $1\le k\le N$ with each value occurring just once (so that $N=nm$ with $n=N$ and $m=1$). In the terminology of lattice rules, the lattice rule is said to be \emph{fully projection regular} (see \cite{lemieux09},  \cite{sloan94}). This property is also clearly illustrated in Figure \ref{fig1b}.\\
\end{example}

\noindent We have that
\[
f_j(z_k) = \int_{[0,1]^{s-1}} f(x_1,\ldots,x_{j-1},z_k,x_{j+1},\ldots, x_s)\,\rd\bsx_{-j}
\]
can be approximated using numerical integration by
\begin{equation} \label{pointwise mean}
        \hat{f}_j(z_k) =
         \frac{1}{m}\sum_{\bsx_i: x_{i,j}=z_k} f(\bsx_i).
\end{equation}
So $\hat{f}_j(z_k)$ is the point-wise mean obtained by averaging out over the $m$ points, for each of whom, $ x_{i,j}=z_k$.
With these approximations to $f_j(z_k)$ for $1\le k\le n$, one can then approximate the shape of $f_j$ by fitting an interpolant  through these $n$ approximations. Note that,   $\hat{f}_j(z_k) $ can be considered as the pointwise mean of the orthogonal projections of $f(\cdot)$ on the $j^{th}$ axis. This is illustrated in Figure \ref{fig1}. An interpolant through the point-wise means of the orthogonal projections of the bi-variate Beta distribution  can approximate the shape of the one dimensional functions reasonably accurately for the $5-$point grid ($m=5$) (Figure \ref{fig1} (b)). But the  rank-$1$ lattice is fully projection regular, i.e., $m=1.$   Although such a property is advantageous for the numerical
integration of integrands over $[0,1]^s$, it is not so advantageous when trying to approximate the shape of the one dimensional functions. We would not expect the approximation to the shape of $f_j$ obtained by fitting an interpolant through the point-wise means (\ref{pointwise mean}) to be an accurate one when $m=1$. For the $32-$point rank-$1$ lattice, the point-wise means of the orthogonal projections of the bi-variate Beta distribution are the projections themselves (Figure \ref{fig1} (c)) and one can see that an interpolant that passes through each one of them would not approximate the shape of the one dimensional function very accurately at all.\\ 

\subsection{Structure of This Paper}
In Section~\ref{approxlds}, we propose a new method that involves use of an integration rule as well as fitting of a polynomial smoothing function to approximate the shape of the one dimensional function. The theoretical results will be presented in Section~\ref{convergence}. In Section~\ref{examples}, we provide some numerical results illustrating the efficiency and accuracy of the approximations produced by our new method as compared to those produced by a grid based method. Finally, we close in Section \ref{discussion} giving a summary of the work and discuss further challenges.

\section{New Method}  \label{approxlds}
Here we propose a method for approximating the shape of the one-dimensional functions
\[
f_j(x)=\int_{[0,1]^{s-1}}f(x_1,\ldots,x_{j-1},x,x_{j+1},\ldots,x_s)\,\rd\bsx_{-j}\quad x\in[0,1],
\]
when, an $s$-dimensional function $f(\boldsymbol{x})$ has been evaluated at $N$ distinct points $\boldsymbol{x}={\boldsymbol{x}_{1},\ldots, \boldsymbol{x}_{N}}$ given by a point set $\mathcal{P}_{N}$. As discussed in Section \ref{1dapprox}, an interpolant through the point-wise means may not provide an accurate approximation when using the QMC integration rules. However, a smoothing function such as a least square polynomial fitted to the projected points  may be a better option. Thus, the basic algorithm we propose is as follows:\\

\begin{underline}{\textbf{Algorithm I ($m > 1$)}} \end{underline}
\begin{enumerate}
    \item {Evaluate the function $f$ at N points $\bsx_{i}.$}
    \item{For $j = 1,\ldots, s,$ do:}
    \begin{enumerate}
	\item {Project the function evaluations $f(\boldsymbol{x}_{i})$ on the $j^{th}$ axis.}
	\item {Fit a polynomial of degree $(n-1)$ to the projections.}
    \end{enumerate}
\item{Repeat for each $j$.}
\end{enumerate}  


\noindent As in Section \ref{1dapprox},  let the components of each $\bsx_i$ be denoted by $x_{i,j}$ for $1\le j \le s, \, 1\le i\le N$.
These components together with the function evaluations may be conveniently represented in a matrix form as
\[
\boldsymbol{\Psi}_{N \times (s+1)} =
\begin{bmatrix}
        x_{1,1}  & x_{1,2} &\cdots  & x_{1,s} & f(\bsx_1)\\
        x_{2,1}  & x_{2,2} &\cdots  & x_{2,s} & f(\bsx_2)\\
        \vdots   & \vdots &\ddots & \vdots    & \vdots\\
        x_{N,1}  & x_{N,2} &\cdots  & x_{N,s} & f(\bsx_N)
\end{bmatrix}.
\]
To approximate the shape of the $j^{th}$ one-dimensional $f_{j}(x)$, we first orthogonally project $ f(\boldsymbol{x}_{i})$ on the $ j^{th}$ axis to obtain 
\[
\psi_j =
\begin{bmatrix}
        x_{1,j}  & f(\bsx_1)\\
        x_{2,j}  & f(\bsx_2)\\
        \vdots   & \vdots\\
        x_{N,j}  & f(\bsx_N)
\end{bmatrix},
\]
More formally, we can write $\psi_{j} = \boldsymbol{\Psi}P_j$,
where $P_{j}$ is the $(s+1)\times 2$ matrix with zeros everywhere
except for ones in the $j$-th position of the first column and the
last position of the second column.

\begin{example}
When $j=2$, we have
\[
        P_2=\begin{bmatrix}
                0 & 1 & 0 & \cdots & 0 & 0\\
                0 & 0 & 0 & \cdots & 0 & 1
        \end{bmatrix}^T.
\]
\end{example}

\noindent Since the spread of the projected function points is not constant (as illustrated by Figure \ref{fig1}), a weighted least square polynomial may be required where the weights are proportional to the variances. However, we prove that in this case, a weighted least square polynomial of degree $(n-1)$ is equal to the ordinary least square polynomial of the same degree. \\

\noindent Let $\underline{M}$ be the design matrix when fitting a least squares polynomial of degree $(n-1)$ through the orthogonal projections of $f(\boldsymbol{x})$ on $x_{j}$. Such a projection has $n$ unique abscissa points $z_{k},\, k=1,\ldots,n,$ as described in Section \ref{1dapprox}. Then $\underline{M}$ is of size $N \times n$, and has a block structure,
\begin{equation*} \label{eq:cbVan}
\underline{M} = 
\begin{bmatrix}
    \boldsymbol{1} & \boldsymbol{t}_1 & \boldsymbol{t}_{1}^{2} & \dots  & \boldsymbol{t}_{1}^{n-1} \\
 \boldsymbol{1} & \boldsymbol{t}_{2} & \boldsymbol{t}_{2}^{2} & \dots  & \boldsymbol{t}_{2}^{n-1} \\
    \vdots & \vdots & \vdots & \ddots & \vdots \\
    \boldsymbol{1} & \boldsymbol{t}_{n} & \boldsymbol{t}_{n}^{2} & \dots  & \boldsymbol{t}_{n}^{n-1} \\
\end{bmatrix},
\end{equation*}
where each element block $\boldsymbol{t}_{k}^{p} \in \underline{M}, (p = 0,\ldots,n-1)$ is an $m \times 1$ column vector containing only the element $z_{k}^{p}$. We can also express $\underline{M}$ as a Kronecker product of the Vandermonde matrix $M$ and the $m \times 1$ column vector of $1's$,
\begin{eqnarray*}
\underline{M} = M \otimes \boldsymbol{1}_{(m \times 1)},
\end{eqnarray*}
where, $M$ is a square Vandermonde matrix of size $n$, which is of full rank and is invertible since all elements $z_{k}$ are unique.\\

\noindent For weighted least squares, we assign a weight $w_{k}$ to all projections corresponding to a unique abscissa point $z_{k}.$  We define the weights matrix $\underline{W}$ of size $N \times n$ by
\begin{equation*}
\underline{W} = 
\begin{pmatrix}
w_1I_{m} &0I_{m}&\cdots&0I_{m} \\
0I_{m}&w_2I_{m}&\cdots&0I_{m}\\
\vdots&\vdots&\ddots&\vdots\\
0I_{m}&0I_{m}&\cdots&w_nI_{m}
\end{pmatrix},
\end{equation*}
where, $I_{m}$ is the identity matrix with size $m \times m$. $\underline{W}$ can also be expressed as a Kronecker product
\begin{equation*}
\underline{W} = W \otimes I_{m},
\end{equation*}
where $W$ is the $n \times n$ diagonal matrix of weights
\begin{eqnarray*}
W = 
\begin{pmatrix}
w_1&0&\cdots&0\\
0 &w_2&\cdots&0\\
\vdots&\vdots&\ddots&\vdots\\
0&0&\cdots&w_n
\end{pmatrix}.
\end{eqnarray*}

\noindent We will make use of the following Kronecker product properties.\\

\noindent \textbf{Lemma 1:} Properties of Kronecker products (\cite{gentle07}) \\
\begin{enumerate}
\item
{\bf{Scalar property}}: For matrices $A$ and $B$, and scalar $k$
\begin{equation*}
(kA) \otimes B = A \otimes (kB) = k(A \otimes B).
\end{equation*}
\item
{\bf{Mixed product property}}: For matrices $A, B, C,$ and $D$, such that $AC$ and $BD$ exist, then
\begin{equation*}
(A \otimes B)(C \otimes D) = AC \otimes BD.
\end{equation*}
\item
{\bf{Inverse property}}: If matrices $A$ and $B$ are invertible, then $(A \otimes B)^{-1}$ exists, and can be expressed as
\begin{equation*}
(A \otimes B)^{-1} = A^{-1} \otimes B^{-1}.
\end{equation*}
\item 
{\bf{Transposition}}: For matrices $A$ and $B$
\begin{equation*}
(A \otimes B)^{T} = A^{T} \otimes B^{T}.
\end{equation*}
\end{enumerate}

\noindent Let $\hat{f}_{j}^{WLS}$ be the weighted least square polynomial approximation of degree $(n-1)$ to the $j^{th}$ one-dimensional function $f_{j}$ and  $\hat{f}_{j}^{LS}$ be the least square polynomial approximation of the same degree. Further, let  $\underline{\hat{f}_{j}^{LS}}$ be the values taken by $\hat{f}_{j}^{LS}$ for the elements in the design matrix $ \underline{M}$. Similarly, $\underline{\hat{f}_{j}^{WLS}}.$

\begin{thm} \label{thm11}
 For any $j \in \{1,\ldots, s\}$, $\hat{f}_{j}^{WLS} = \hat{f}_{j}^{LS}.$
\end{thm}
\begin{proof}  We have
\begin{eqnarray}  \label{ls1}
\nonumber\underline{\hat{f}_{j}^{LS}} &=& \underline{M}(\underline{M}^{T}\underline{M})^{-1} \underline{M}^{T} \boldsymbol{f} \\  
\nonumber 
 &=& (M \otimes \boldsymbol{1}) \left[(M \otimes \boldsymbol{1})^{T} (M \otimes \boldsymbol{1})\right]^{-1} (M \otimes 
\boldsymbol{1})^{T} \boldsymbol{f}\\ \nonumber
&=&( M(M^{T}M)^{-1}M^{T}) \otimes (\boldsymbol{1}(\boldsymbol{1}^{T}\boldsymbol{1})^{-1}\boldsymbol{1}^{T}) \boldsymbol{f}. \\ 
\nonumber
\end{eqnarray}
Since $M$ is a square Vandermonde matrix and invertible, and $\boldsymbol{1}^{T}\boldsymbol{1} = m$, we have
\begin{eqnarray}
\nonumber \underline{\hat{f}_{j}^{LS}} &=& (MM^{-1}(M^{T})^{-1}M^{T}) \otimes (\dfrac{1}{m} \boldsymbol{1} \boldsymbol{1}^{T}) \boldsymbol
{f} \\ \nonumber
&=& \dfrac{1}{m} I_n \otimes (\boldsymbol{1}\boldsymbol{1}^{T})\boldsymbol{f}. \\ 
\end{eqnarray}
We have
\begin{eqnarray} 
\nonumber \underline{\hat{f}_{j}^{WLS} }&=& \underline{M} (\underline{M}^{T}\underline{W}\underline{M})^{-1} \underline{M}^{T} 
\underline{W} \boldsymbol{f}\\
\nonumber &=& (M \otimes \boldsymbol{1}) \left((M \otimes \boldsymbol{1})^{T} (W \otimes I_{m}) (M \otimes \boldsymbol{1}) 
\right)^{-1} (M \otimes \boldsymbol{1})^{T} (W \otimes I_{m}) \boldsymbol{f}\\
\nonumber &=&( M(M^{T}WM)^{-1}M^{T}W) \otimes (\boldsymbol{1}(\boldsymbol{1}^{T} I_m \boldsymbol{1})^{-1} \boldsymbol{1}^{T} I_m) 
\boldsymbol{f}. \\ \nonumber
\end{eqnarray}
Since $W$ is also square and invertible ($W$ is a diagonal matrix, with $w_{i,i} > 0$), and $(\boldsymbol{1}^{T} I_m 
\boldsymbol{1}) = m$, we have
\begin{eqnarray}
\nonumber\underline{ \hat{f}_{j}^{WLS}} &=&( MM^{-1}W^{-1}(M^{T})^{-1}M^{T}W) \otimes (\dfrac{1}{m} \boldsymbol{1}\boldsymbol{1}^{T} 
I_m) \boldsymbol{f} \\
\nonumber &=& (I_n W^{-1} I_n W) \otimes (\dfrac{1}{m} \boldsymbol{1}\boldsymbol{1}^{T} I_m) \boldsymbol{f} \\
\nonumber &=& \dfrac{1}{m} I_n \otimes (\boldsymbol{1}\boldsymbol{1}^{T})\boldsymbol{f} = \hat{f}_{j}^{LS}(z_k).
\end{eqnarray}
\end{proof}

\noindent We can further show that $\hat{f}_{j}^{LS}$ will pass through $\hat{f}_j(z_k)$ for each $k$.

\begin{thm} \label{thm12}
 For any $j \in \{1,\ldots, s\}$, $\hat{f}_{j}^{LS}$ will pass through $\hat{f}_j(z_k)$ for  $1 \le k\le n$.
\end{thm}
\begin{proof} Using Equation  (\ref{ls1}) we have that
\begin{eqnarray*}
\nonumber \underline{\hat{f}_{j}^{LS}} &=& \frac{1}{m} (I_{n} \otimes \boldsymbol{1}\boldsymbol{1}^{T}) \boldsymbol{f}\\
&=& \frac{1}{m}
\begin{pmatrix}
    J_{m} & 0_{m} &  \dots  & 0_{m} \\
    0_{m} & J_{m} & \dots  & 0_{m}\\
    \vdots & \vdots  & \ddots & \vdots \\
    0_{m} & 0_{m} & \dots  & J_{m} \\
\end{pmatrix}
\begin{pmatrix}
\boldsymbol{f}_1 \\
\boldsymbol{f}_2\\
\vdots\\
\boldsymbol{f}_n\\
\end{pmatrix} 
=
\begin{pmatrix}
\tilde{f}_{j,1}\\
\tilde{f}_{j,2}\\
\vdots\\
\tilde{f}_{j,n}\\
\end{pmatrix},
\end{eqnarray*}
where each element $J_{m}$ or $0_{m}$ is a square matrix of size $m \times m$ that contains all 1's or all 0's respectively and  
$\boldsymbol{f}_{k},\; k=1,\ldots,n$ is the $m \times 1$ vector of function evaluations $f(\boldsymbol{x})$  corresponding to 
$z_{k}$.
 \end{proof}

%

\noindent For fully projection regular point sets such as many of the LDS, including the rank-$1$ lattice rules, $m=1$ and using Algorithm I in such cases will imply fitting a polynomial of degree $(N-1)$  passing through all of the $N$ function projections. Such a polynomial will not approximate the desired shape accurately. Here, we propose a partitioning approach to overcome this problem. Suppose we partition $[0,1]$ into $n$ partitions, with breakpoints given by $0=z_0< z_1 < z_2< \ldots < z_{n-1}<z_n=1$. As above, we assume that the total number of points $N$ factorises as $N=nm$. Further, we assume the points are such that there are exactly $m$ points whose $j$-th component belongs to $[z_k,z_{k+1})$ for $0\le k\le n-1$. Note that these assumptions are not necessary for the validity of the theory, instead, they have been made only to simplify the notation. We provide below the modified algorithm with a partitioning step.\\ 

\noindent \begin{underline}{\textbf{Algorithm II ($m = 1$)}} \end{underline}
\begin{enumerate}
    \item {Evaluate the function $f$ at N points $\bsx_{i}.$}
    \item{For $j = 1,\ldots, s,$ do:}
    \begin{enumerate}
	\item {Project the function evaluations $f(\boldsymbol{x}_{i})$ on the $j^{th}$ axis.}
	\item{Partition $[0,1]$ into $n$ partitions, with breakpoints given by $0=z_0< z_1 < z_2< \ldots < z_{n-1}<z_n=1$.}
	\item {Fit a polynomial of degree $(n-1)$ to the projections.}
    \end{enumerate}
\item{Repeat for each $j$.}
\end{enumerate}  

\noindent Similar to (\ref{pointwise mean}), one can calculate
\begin{equation} \label{partition}
	\tilde{f}_{j,k} (z_k) =
	 \frac{1}{m}\sum_{\bsx_i: x_{i,j}\in[z_k,z_{k+1})} f(\bsx_i).
\end{equation}
Let $\tilde{f}_{j}^{LS}$ be the least square polynomial of degree $(n-1).$ Then, we can show that $\tilde{f}_{j}^{LS}$ will pass through $\tilde{f}_{j,k} (z_k)$ for each $k$.

\begin{thm} \label{thm122}
 For any $j \in \{1,\ldots, s\}$, $\tilde{f}_{j}^{LS}$ will pass through $\tilde{f}_{j,k}$ for  $0\le k\le n-1$.
\end{thm}
\begin{proof} The proof is similar to that of Theorems \ref{thm11} and \ref{thm12}.
\end{proof}

\noindent MC integration rules generate points that are fully projection regular $w.p.$ (\emph{with probability}) $\, 1.$ Therefore Algorithm II, approximation (\ref{partition}) and Theorem \ref{thm122} are also applicable when the function has been evaluated using a random point set.

\section{Convergence theorems} \label{convergence}

\subsection{For point sets where $m>1$}

The new approach described in the previous section essentially involves evaluating $f$ on a set of $N$ points in $[0,1]^s$ and then approximating the  one-dimensional function $f_j$ by fitting a least square polynomial through the orthogonal projections $ f(\bsx_i)$  of $f(\cdot)$ on the $j^{th}$ axis. Theorem \ref{thm12} proves that $\hat{f}_{j}^{LS}$ passes through the $n$ point-wise means $\hat{f}_j(z_k)$. This implies that this approach is equivalent to the interpolating polynomial approach where a polynomial of degree $(n-1)$ is fitted to $n$ function evaluations. Therefore the convergence properties can be studied using the relevant literature in numerical analysis.
We assumed that there were $N=n \times m$ points in $[0,1]^s$ such that
$f_j$ is approximated at $n$ distinct points $z_k$, $1\le k\le n$, and that for each unique value of $z_k$, there is a subset of $m$ points
whose $k$-th co-ordinate is equal to $z_k$.\\

The choice of the points $z_k$ is crucial and determines the convergence properties and the computational efficiency as discussed below.  The next theorem gives the convergence result when the $z_k$ are equidistant points (in a grid).

\begin{thm} \label{thm3}
Suppose that $f_j$ is infinitely differentiable such that
\[
\max_{\xi \in [0,1]} |f_j^{(n)}(\xi)| \leq C,\, \forall n,
\]
for some $ C<\infty$ such that $\frac{C}{(n-1)^n}\ll 1 ,\, \forall n$.
If the $z_k$ are equidistant points,
then $\hat{f}_{j}^{LS}\to f_j$ as $m\to\infty$ and $n\to\infty$.
\end{thm}

\begin{proof}

As $m\rightarrow \infty$,
\begin{equation}\label{eqn_converge}
\hat{f}_j(z_k) = \frac{1}{m}\sum_{\bsx_i: x_{i,j}=z_k} f(\bsx_i)
\rightarrow f_j(z_k).
\end{equation}
Equation (\ref{eqn_converge}) holds due to the Koksma-Hlawaka inequality
(\cite{nied92}) if the $\bsx_i$ are sampled using a grid.

\noindent For the interpolating polynomial of degree $n-1$,
it follows from a standard result in approximation theory (see for example, \cite{cheney99}, \cite{kress98}), that
\[
\max_{z\in[0,1]} |f_j(z) -\hat{f}_{j}^{LS}(z)| \leq \max_{\xi\in [0,1]} \frac{|f_j^{(n)}(\xi)|}{n!} \max_{z\in[0,1]} \prod_{k=1}^{n} |z-z_k|.
\]
This implies that
\begin{equation}\label{bound}
\max_{z\in[0,1]} |f_j(z) - \hat{f}_{j}^{LS}(z)| \leq \frac{C}{n!} \max_{z\in[0,1]} \prod_{k=1}^{n} |z-z_k|.
\end{equation}

It is known that (see for example, \cite{cheney99}) that if the $n$ points $z_k$ are equidistant on $[0,1]$, then
\[
\max_{z\in[0,1]} \prod_{k=1}^{n} |z-z_k| \leq \frac{(n-1)!}{4} \left(\frac{1}{n-1}\right )^{n}.
\]
From (\ref{bound}), we then have
\[
\max_{z\in[0,1]} |f_j(z) - \hat{f}_{j}^{LS}(z)|
\leq \frac{C}{4n(n-1)^n}.
\]
The assumption that $\frac{C}{(n-1)^n}\ll 1$ for all $n$
then implies that as $m \rightarrow \infty$ and $n \rightarrow \infty$, $\hat{f}_{j}^{LS}
\to f_j$.
\end{proof}

If the function $f_{j}$ is $n$ times differentiable then the result in Theorem \ref{thm3} indicate that the approximation obtained using $\hat{f}_{j}^{LS}$ will still be good as long as the derivatives are sufficiently bounded. 

\subsection{For fully projection regular point sets where $m=1$}

Theorem \ref{thm3} provides the conditions under which $\hat{f}_{j}^{LS} \to f_j$ for grids constructed using equidistant points. Now, we show that the polynomial approximation will converge to the shape of the true one dimensional function  if the function was explored using LDS instead of a grid. 

\begin{thm} \label{thm5}
Let $h_k = z_{k+1}-z_k$  using the partitions defined in Algorithm II, and points $\bsx_{i}$ sampled using a QMC integration rule. 
If $\tilde{f}_{j,k}$ is as given in (\ref{partition}), then
$\tilde{f}_{j,k}\to f_j(z_k)$ as $m\to\infty$ and $h_k\to 0$.
\end{thm}

\begin{proof}
One may consider $\tilde{f}_{j,k}$ as an approximation to the integral
\begin{equation}
\label{eqn_converge2}
\frac{1}{h_k} \int_{[0,1]^{s-1}} \int_{z_k}^{z_{k+1}} f(\bsx)\,\rd x_j\,\rd\bsx_{-j}.
\end{equation}
As $m\to\infty$, $\tilde{f}_{j,k}$ converges to this integral due to the Koksma-Hlawaka inequality
(\cite{nied92}). For the integral in (\ref{eqn_converge2}), we can swap the order of
integration by Fubini's theorem since $f$ is integrable and
Lebesgue measure is a $\sigma$-finite measure. So the integral becomes
\[
\frac{1}{h_k}\int_{z_k}^{z_{k+1}}\int_{[0,1]^{s-1}} f(\bsx)\,\rd\bsx_{-j}\,\rd x_j=
\frac{1}{h_k}\int_{z_k}^{z_{k+1}}f_j(x_j)\,\rd x_j.
\]
Letting $h_k\to 0$, it follows from the definition of derivative that
this integral converges to $f_j(z_k)$.
\end{proof}


\begin{thm} \label{thm6}
Suppose that $f_j$ is infinitely differentiable such that
\[
\max_{\xi \in [0,1]} |f_j^{(n)}(\xi)| \leq C,\, \forall n,
\]
for some $ C<\infty$ such that $\frac{C}{(n-1)^n}\ll 1 ,\, \forall n$.
If the $z_k$ are equidistant points and points $\bsx$ sampled using a QMC integration rule,
then $\tilde{f}_{j}^{LS}\to f_j$ as $m\to\infty$ and $n\to\infty$.
\end{thm}

\begin{proof}
The result follows from Theorems \ref{thm3}, \ref{thm122} and \ref{thm5}.
\end{proof}

Note that if the function $f_{j}$ is $n$ times differentiable then the results in Theorem~\ref{thm6} indicate that the approximation obtained using $\tilde{f}_{j}^{LS}$ will still be good as long as the derivatives are sufficiently bounded.

\subsection{For random point sets}

As pointed out in Section \ref{approxlds}, Algorithm II, approximation (\ref{partition}) and Theorem \ref{thm122} are also applicable when the function has been evaluated using a random point set. We provide the corresponding result for this case. 

\begin{thm} \label{thm7}
Let $h_k = z_{k+1}-z_k$  using the partitions defined in Algorithm II, and points $\bsx$ sampled using a MC integration rule. 
If $\tilde{f}_{j,k}$ is as given in (\ref{partition}), then
$\tilde{f}_{j,k}\to f_j(z_k)\, w.p.\, 1$ as $m\to\infty,$ and $h_k\to 0$.
\end{thm}

\begin{proof}
Proof is similar to Theorem \ref{thm5} except that as $m\to\infty$, $\tilde{f}_{j,k}$ converges to the integral (\ref{eqn_converge2}) $w.p.\, 1$ because of the law of large numbers. The rest of the proof is exactly the same.
\end{proof}

\begin{thm} \label{thm8}
Suppose that $f_j$ is infinitely differentiable such that
\[
\max_{\xi \in [0,1]} |f_j^{(n)}(\xi)| \leq C,\, \forall n,
\]
for some $ C<\infty$ such that $\frac{C}{(n-1)^n}\ll 1 ,\, \forall n$.
If the $z_k$ are equidistant points and points $\bsx$ sampled using a MC integration rule,
then $\tilde{f}_{j}^{LS}\to f_j$ as $m\to\infty$ and $n\to\infty$.
\end{thm}

\begin{proof}
The result follows from Theorems \ref{thm3}, \ref{thm122} and \ref{thm7}.
\end{proof}

Note that if the function $f_{j}$ is $n$ times differentiable then the results in Theorem~\ref{thm8} indicate that the approximation obtained using $\tilde{f}_{j}^{LS}$ will still be good as long as the derivatives are sufficiently bounded.

\section{Numerical Examples}\label{examples}

The algorithms proposed in Section \ref{approxlds} can be used when a function is explored using a grid, MC or QMC integration rules. However, because this work was motivated by the need to develop a method for QMC integration rules (and no other method exists, to our best knowledge), we focus on QMC integration rules in the examples below. Wherever possible, we also compare the results against those obtained using a grid. Since this problem was motivated by a possible application in the Bayesian statistical inference, we  illustrate  using a few standard probability distributions. 

The integration rules used are known as Korobov lattice rules. These
are rank-$1$ lattice rules in which the generating vector $\bsz$ in
(\ref{gen}) is of the form
\[
\bsz=(1,\alpha,\alpha^2,\ldots,\alpha^{s-1}),
\]
where $\alpha$ is an integer in $\{1,2,\ldots, N-1\}$. Appropriate
choices of $\alpha$ may be found by using the Lattice Builder software (see
\cite{lecuyer16}).

\subsection{Exponential distribution}

Most statistical distributions are smooth with bounded derivatives and therefore satisfy the smoothness requirements of Theorems \ref{thm3}, \ref{thm6} and \ref{thm7}. Here, we illustrate how the exponential distribution, for example, satisfies these smoothness conditions. The Exponential distribution is slightly different since the derivative does not exist at zero. However, here we show that it still satisfies the smoothness conditions imposed by Theorems \ref{thm3}, \ref{thm6} and \ref{thm7}. Suppose that the $j$-th one dimensional distribution is exponential with parameter $\lambda$. Then we have that,
\[
f_j(x) = \lambda e^{-\lambda x};
\]
the $n^{th}$ derivative is given by
\[
f_j^{(n)}(x) = (-1)^{n}\lambda^{n+1} e^{-\lambda x},
\]
and
\[
\sup_{x} |f_j^{(n)}(x)| = \lim_{x \rightarrow 0+} |f_j^{(n)}(x)| = \lambda^{n+1}.
\]
We assume here that the interval of interest is $[0,b)$ for some $b<\infty,$ $b$ large enough so that $\int_{0}^{b} f_j(x)\, dx \approx 1.$  Note that the convergence results proved in Section \ref{convergence} are applicable here since the function can be linearly transformed to be defined over $[0,1]$. Then, $\exists n'>0$ and $c<1$ such that $\forall n> n'+1, \, \frac{b}{n-1} \leq \frac{1}{nc} < 1$. Further, for any $\lambda < \infty, \, \exists n'' > n'$ such that, $\forall n>n'', \lambda^{n+1}\left (\frac{1}{nc}\right )^{n} \ll 1$.

\begin{figure} [h]
\includegraphics[scale=0.7]{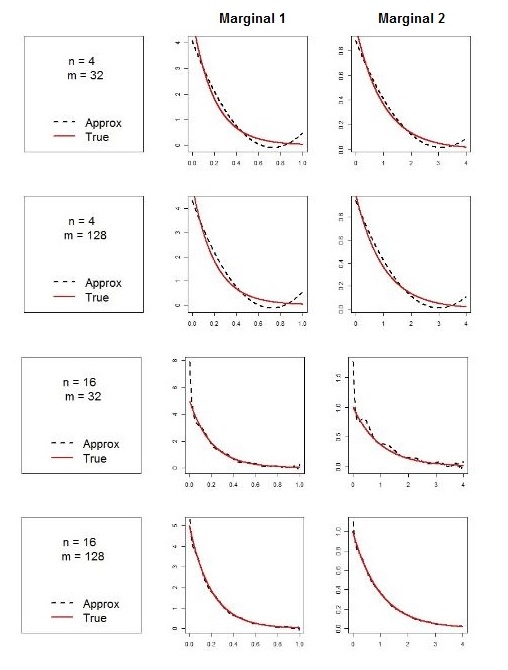}
\caption[]{ Least squares approximation to the Exponential marginals using Korobov lattices as $n$ and $m$ increase.}
\label{fig_Exp}
\end{figure}

Thus, it can be seen that conditions for Theorem \ref{thm6} are satisfied and $\tilde{f}_{j}^{LS} \to f_j$ as $m\rightarrow \infty$ and $n\rightarrow \infty$. This is illustrated in Figure \ref{fig_Exp}. Here, the joint distribution is bi-variate and is a product of two Exponential distributions. We find the least squares approximations to the marginals using Korobov lattices with different $n$ and $m$, the convergence is achieved as they both increase.

\subsection{ Multi-modal and skewed distributions}

Figures \ref{fig_multi} and \ref{fig_beta} illustrate that a grid is quite inefficient at accurately capturing the shape of the distribution even in low dimensional problems, especially when it is multi-modal or heavily skewed. Here, we consider a multi-modal distribution and the Beta distribution, in four variables,  and try to approximate the shape of the marginals using the grid points (and fitting the interpolant through pointwise means) as well as using LDS points and our new method of fitting the least squares polynomials of degree $(n-1)$ through the orthogonal projections of the joint distribution on the marginals proposed in this paper.

Figure \ref{fig_multi} shows that the marginals approximated using the Korobov lattice with $4096$ points are very accurate whereas the approximation using an $8-$point grid with the same number of points ($8^4 = 4096$) is not as accurate. Figure \ref{fig_beta} illustrates that the approximations to Beta marginals using a $1024$ point Korobov lattice are much more accurate than the approximations obtained using grids with $6^4 = 1296$ or even $8^4 = 4096$ points. Thus, using LDS enables efficient and more accurate approximation of the shape of  the one-dimensional distributions.

\begin{figure} [h]
\includegraphics[scale=0.35]{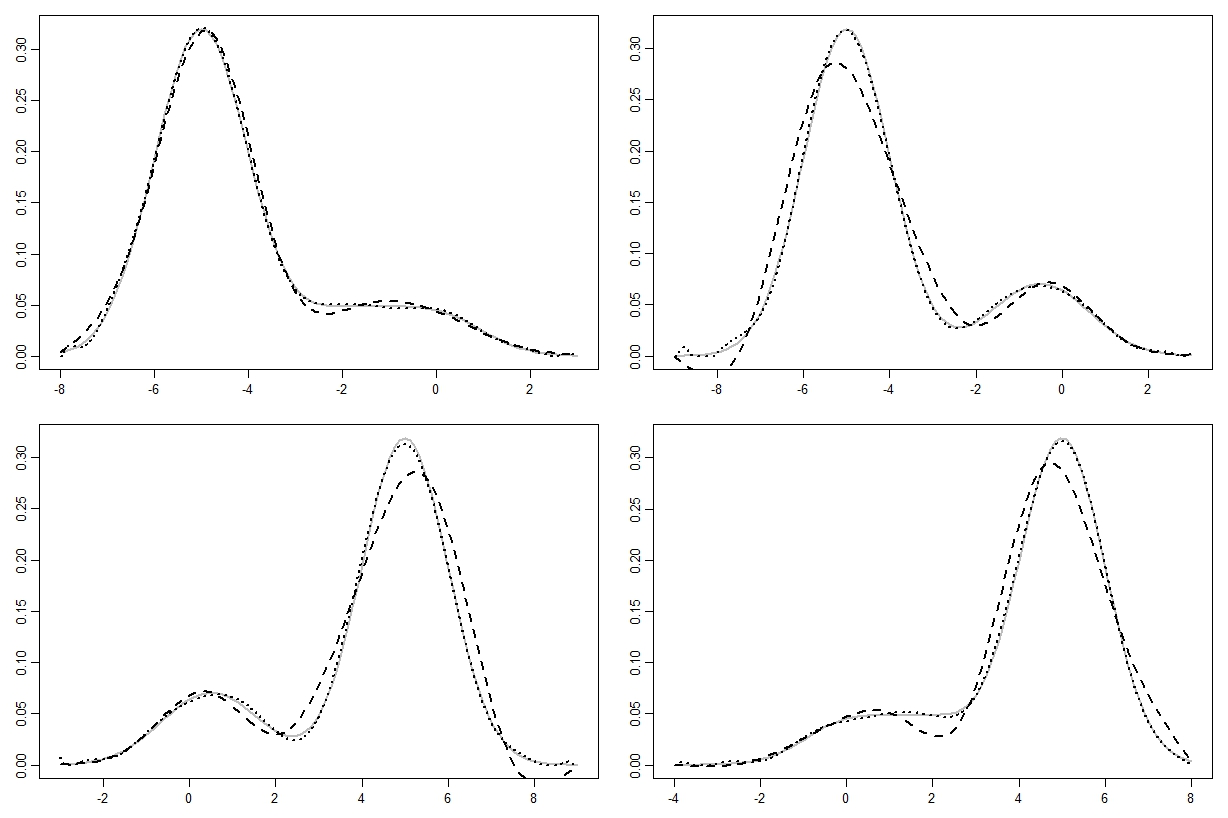}
\caption[]{Approximating marginals of a four-dimensional multi-modal distribution (line) using: Korobov lattices with $4096$ points (dotted) and an $8-$point grid with $4096$ points (dashed).}
\label{fig_multi}
\end{figure}

\begin{figure} [h]
\includegraphics[scale=0.35]{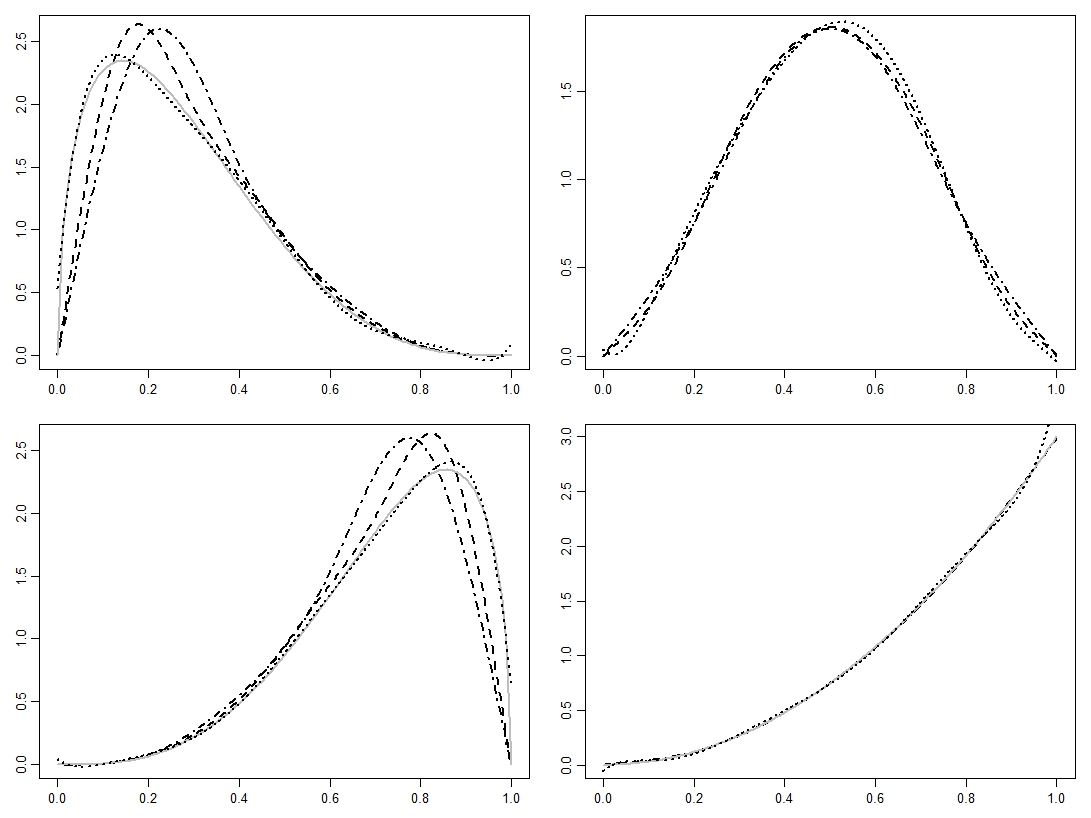}
\caption[]{Approximating marginals of a four-dimensional Beta distribution (line) using: Korobov lattices with $1024$ points (dotted), a $6-$point grid with $1296$ points (dash-dotted) and a $8-$point grid with $4096$ points (dashed).}
\label{fig_beta}
\end{figure}

\subsection{ High-dimensional distributions} \label{gamma-ex}

To illustrate the real computational benefit of using low discrepancy sequences, we consider two distributions of dimensions $10$ and $12$ respectively. These distributions have been generated as products of independent Gamma distributions with different parameters. A $5$-point grid will require $5^{10} = 9,765,625$ points in $10$ dimensions and $244,140,625$ points in $12$ dimensions and will likely still yield inaccurate estimates, as illustrated by an inability of $n-$point grids to capture various shapes when $n$ is small in Figures \ref{fig_multi} and \ref{fig_beta}. 

\begin{figure}
\includegraphics[scale=0.36]{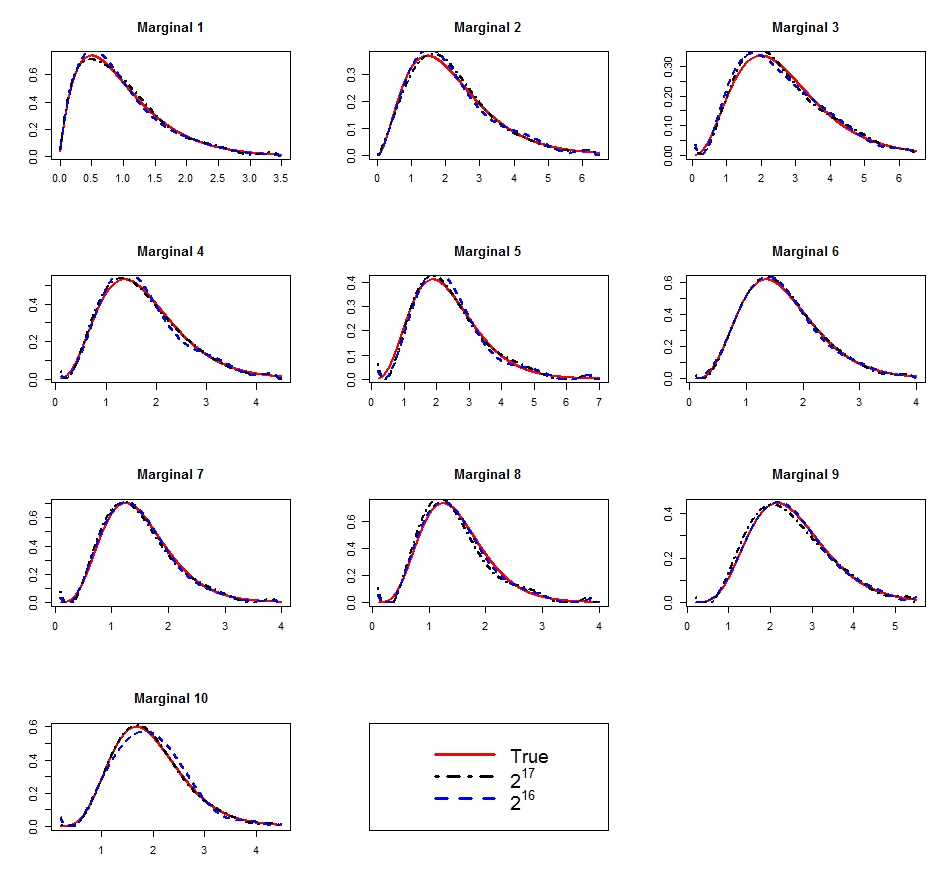}
\caption[]{$10$-dimensional Gamma using Korobov lattice with i) $2^{16} = 65,536$  (dashed) and ii) $2^{17}=131,072$ points (dotted).}
\label{fig_10d}
\end{figure}

Figure \ref{fig_10d} shows that for $s=10$, very accurate estimates can be obtained using LDS with as little as $2^{16}$ points ($150$ times fewer than a $5$-point grid). Although estimates obtained using $2^{17}$ points are even more accurate, the difference between the two is very small suggesting that our estimates have started to converge to the true marginals. For $12$-dimensional Gamma, $2^{16}$ points give reasonably accurate estimates and the convergence is achieved by $2^{19} (= 524,288)$ points as can be seen in Figure \ref{fig_12d}. However, this is negligible compared to the $244$ million points required for a $5$-point grid.

\begin{figure}
\includegraphics[scale=0.36]{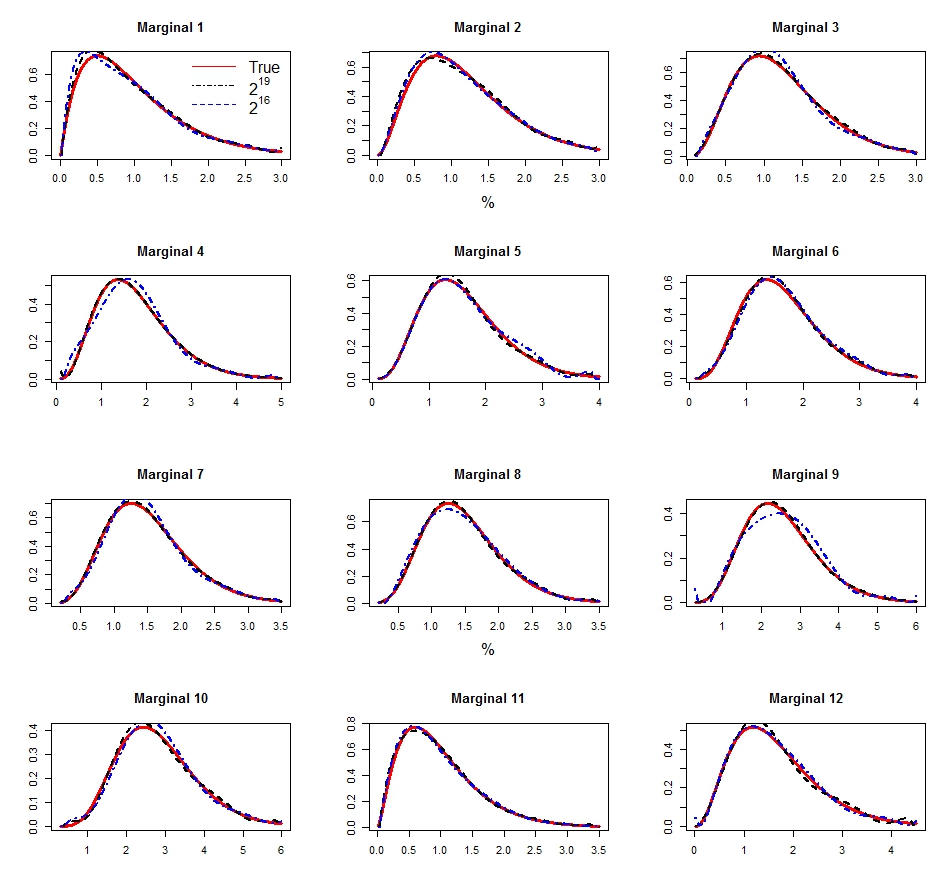}
\caption[]{$12$-dimensional Gamma using Korobov lattice with i) $2^{16} = 65,536$ (dashed) and ii) $2^{19}=524,288$ points (dotted).}
\label{fig_12d}
\end{figure}

\section{Summary and Discussion}\label{discussion}

This paper proposes a new method to approximate the shape of one dimensional functions $f_{j},$ where,  $f_j$ is the function obtained by integrating out all the variables of an $s-$dimensional function $f$ except the $j$-th one and where the function has been explored using a point set. Not only is this method easy and computationally efficient but also, it can be used when the function is evaluated using the grid, the MC or the QMC integration rules. To our best knowledge, a formal method to solve this problem has not been proposed yet, especially for QMC integration rules.  The method  uses a least squares polynomial smoother. We propose two algorithms - two versions of the method -  one where the point set used are fully projection regular (or fully projection regular $w.p.\, 1$, in case of MC rules) and the other when this is not the case. We prove the convergence properties for both these algorithms. We show that implementing our new method using LDS points only requires $O(mn)$ function evaluations, compared to the traditional grid based approaches that require $O(n^{s})$ function evaluations. Typically, $m < n^{(s-1)}$ and therefore implementing our new method using LDS points is computationally more efficient than using an $n$ point grid. Further, the examples illustrate that our method also produces more accurate approximation than using the traditional grid based approach.\\

The need to develop such a method was motivated by a potential application in  Bayesian statistics, specifically, in computational methods that explore the posterior distribution using a set of deterministic point sets as discussed in Section \ref{motivate}.  However, practical challenges will need to be overcome before the method developed here can be incorporated within the computation Bayesian methods. For instance, the proposed method provides asymptotic guarantees as  the number of points and the degree of the polynomial go to infinity. However, it cannot specify the number of points and the degree of the polynomial needed to achive a reasonable approximation for a given function or indeed for a wide range of functions (class of all continuous probability distributions, for example). Thus further work will be required to develop a method that can potentially improve the computational efficiency of Bayesian methods using QMC integration rules.\\

However, to the best of our knowledge, this paper presents the first formal method developed to approximate the shape of the one-dimensional function obtained by integrating out all other variables using LDS.

\section*{Acknowledgements}
Paul Brown's research has been funded by a University of Waikato doctoral scholarship.



\end{document}